\documentclass[a4paper,12pt]{article}
\usepackage{amsmath,amsthm,amssymb,amsfonts}
\usepackage{bbm,bm}
\usepackage{latexsym}
\usepackage{mathrsfs}
\usepackage{threeparttable}
\usepackage{tabularx}
\usepackage{booktabs}
\usepackage{graphicx,subfigure}
\usepackage{color}
\usepackage{indentfirst}
\usepackage{geometry}
\usepackage{caption}
\usepackage{lineno}
\usepackage{abstract}
\usepackage{enumerate,mdwlist}
\usepackage[numbers,sort&compress]{natbib}
\usepackage[colorlinks,
            linkcolor=blue, 
            citecolor=red  
            ]{hyperref}
\usepackage{epstopdf}

\def \[{\begin{equation}}
\def \]{\end{equation}}

\newtheorem{thm}{Theorem}[section]

\newtheorem{lem}[thm]{Lemma}
\newtheorem{cor}[thm]{Corollary}

\begin{document}

\setlength{\baselineskip}{20pt}
\begin{center}{\Large \bf The maximum matching extendability and factor-criticality of 1-planar graphs}

\vspace{4mm}

{Jiangyue Zhang, Yan Wu, Heping Zhang\footnote{The corresponding author.}
\renewcommand\thefootnote{}\footnote{E-mail addresses: jyzhang20@lzu.edu.cn(J. Zhang), wuyan20@lzu.edu.cn(Y. Wu), zhanghp@lzu.edu.

cn(H. Zhang).}}

\vspace{2mm}

\footnotesize{ School of Mathematics and Statistics, Lanzhou University, Lanzhou, Gansu 730000, P. R. China}

\end{center}
\noindent {\bf Abstract}:  A graph is $1$-$planar$ if it can be drawn in the plane so that each edge is crossed by at most
one other edge. Moreover, a 1-planar graph $G$ is $optimal$ if it satisfies $|E(G)|=4|V(G)|-8$. 
J. Fujisawa et al. \cite{16} first considered matching extension of optimal 1-planar graphs,   obtained that each optimal 1-planar graph of even order is 1-extendable  and characterized 2-extendable optimal 1-planar graphs and 3-matchings extendable to perfect matchings as well. In this short paper, we prove that no optimal $1$-planar graph is 3-extendable.  Further we mainly obtain that no 1-planar graph is 5-extendable by the discharge method and also  construct a 4-extendable 1-planar graph.  Finally we get that no 1-planar graph is 7-factor-critical and no optimal 1-planar graph is 6-factor-critical.

\vspace{2mm} \noindent{\bf Keywords}: $n$-extendable graph; $k$-factor-critical graph; 1-planar graph; optimal 1-planar graph
\vspace{2mm}


 {\setcounter{section}{0}
\section{Introduction}\setcounter{equation}{0}
All graphs considered here are finite, simple and undirected. We use $V(G)$ and $E(G)$ to denote the vertex set and the edge set of a graph $G$ respectively.  The order of a graph refers to the number of vertices. For a vertex $v$ in $G$, let ${d}_{G}(v)$ denote the degree of a vertex $v$ in $G$, the number of edges incident with $v$. Let $\delta(G)$ denote the minimum degree in $G$. A \emph{matching} $M$ of a graph $G$ is a subset of $E(G)$ such that any two edges of $M$ have no end-vertices in common. A matching of $k$ edges is called a \emph{$k$-matching}. A \emph{perfect matching} of a graph is a matching covering all vertices of the graph.

In 1980 M.D. Plummer  \cite{2} introduced the concept of  $n$-extendable graphs. For an integer $n\geq 0$,  a connected graph $G$ with at least $2n+2$ vertices is said to be {\em $n$-extendable} if it admits an $n$-matching and each $n$-matching  can be a subset of  a perfect matching of $G$.  
Matching extendability of graphs was widely investigated; see two surveys   \cite{PL96, PL08} and a book \cite{YL09}. Plummer  proved that every $n$-extendable graph is $(n-1)$-extendable whenever $n\geq 1$ and the following basic result.

\begin{lem}[\cite{2}]\label{1.1} If  $G$ is $n$-extendable, then $G$ is $(n+1)$-connected.
\end{lem}
%

Plummer  studied the matching extendability of planar graphs and obtained

\begin{thm}[\cite{4}]\label{thm1.3} No planar graph is 3-extendable.
\end{thm}

\begin{thm}[\cite{PL92}]\label{2-ext}Every 5-connected planar graph of even order is 2-extendable.
\end{thm}

For any surface $\Sigma$, Plummer \cite{PL88} also posed the problem of determining the least
integer $\mu(\Sigma)$ such that no graph $G$ embedded in $\Sigma$ is $\mu(\Sigma)$-extendable. $\mu(\Sigma)$ is called the matching extendability of the surface $\Sigma$. So by Theorems \ref{thm1.3} and \ref{2-ext}, $\mu(S_{0})=3$ for the sphere $S_0$. Later, N. Dean \cite{9} completely solved this problem by obtaining the following result:

\begin{thm}[\cite{9}]\label{thm1.4} If $\Sigma$ is any surface (orientable or non-orientable) other than the sphere, then $\mu(\Sigma) =2+\lfloor \sqrt{4-2\chi(\Sigma)}\rfloor$, where $\chi(\Sigma)$ is the Euler characteristic of the surface $\Sigma$.
\end{thm}

A graph is 1-$planar$ if it can be drawn in the plane so that each edge is crossed by at most one other edge. Unless otherwise stated, we consider that a given 1-planar graph is already drawn on the plane (or sphere). The notion of a 1-planar graph was first introduced by G. Ringel \cite{11}. To now most researches on 1-planar graphs have focused on graph coloring and some structures. For structure of 1-planar graphs, I. Fabrici and T. Madras \cite{14} showed that every 1-planar graph $G$ has at most $4|V(G)|-8$ edges.  A 1-planar graph $G$ is called an {\em optimal 1-planar graph} if it satisfies $|E(G)|=4|V(G)|-8$. They also  got

\begin{lem}[\cite{14}]\label{1.7} Every 1-planar graph contains a vertex of degree at most 7.
\end{lem}

Motivated by Plummer's work it is natural to consider matching extension of 1-planar graphs.  However  J. Fujisawa et al. \cite{16} pointed out that an obstacle for such a research is the fact that the operation of edge contraction does not preserve the 1-planarity in general.  In 2018 they eliminated this difficulty for optimal 1-planar graphs,  and got a series of results on  the matching extendability of optimal 1-planar graphs like the case of planar graphs.

In order to state such results, we need the following notations. An edge in a 1-planar graph $G$ is called crossing if it crosses with another edge, and non-crossing otherwise. A cycle $C$ in a connected graph $G$ is said to be {\em separating} if $G-V(C)$ is  disconnected. 
A separating cycle $C$ of a 1-planar graph is called a {\em barrier cycle} if each edge of $C$ is non-crossing, $G-V(C)$ consists of two odd components which lie in the interior and exterior of  $C$ respectively.

\begin{thm}[\cite{16}]\label{thm1.8} Every optimal 1-planar graph $G$ of even order is 1-extendable.
\end{thm}

\begin{thm}[\cite{16}]\label{thm1.9} An optimal 1-planar graph $G$ of even order is 2-extendable unless $G$ contains a barrier cycle of length 4.
\end{thm}

\begin{thm}[\cite{16}]\label{thm1.13} Let $G$ be a 5-connected optimal 1-planar graph of even order and $M$ be a matching of $G$ with $|M|=3$. Then $M$ is extendable unless $G$ contains a barrier cycle $C$ of length 6 such that $V(M)=V(C)$.
\end{thm}

Theorem \ref{thm1.9} implies that  a 5-connected optimal 1-planar graph $G$ of even order is 2-extendable since it cannot contain a barrier cycle of length 4.
In this paper, we  get the following result for any optimal 1-planar graphs.

\begin{thm}\label{thm1.15} No optimal 1-planar graph is 3-extendable.
\end{thm}

By Theorems \ref{thm1.13} and  \ref{thm1.15}, we can obtain the following corollary.

\begin{cor}\label{cor1.16} Any 5-connected optimal 1-planar graph $G$ of even order contains a barrier cycle $C$ of length 6.
\end{cor}

Next we mainly  consider the maximum matching extendability of 1-planar graphs. Since $n$-extendable graphs are $(n+1)$-connected, Lemma \ref{1.7} implies that no 1-planar graph is 7-extendable. However we  prove that no 1-planar graph is 5-extendable by combining the discharge method and Dean's lemma as follows. We construct a 4-extendable 1-planar graph (see Fig. \ref{p10}), which shows that such result is best possible.

\begin{thm}\label{thm1.14} No 1-planar graph is 5-extendable.
\end{thm}

The remaining sections of this paper are organized   as follows. In Section 2, we  give the proof of Theorem \ref{thm1.15}  via Suzuki's relation between optimal 1-planar graphs and quadrangulations on the sphere. In Section 3 we use  Dean's lemma and  discharging method to prove Theorem \ref{thm1.14}. At the end of this section, we give a 4-extendable 1-planar graph. In Section 4, as remarks we show that no 1-planar graph is 7-factor-critical and no optimal 1-planar graph is 6-factor-critical. Some examples show such results are best possible.


\section{Proof  of Theorem \ref{thm1.15}}

A quadrangulation of the sphere is a simple graph embedded on the sphere with no crossing point such that each of its faces is bounded by a cycle of length 4. If we remove all the crossing edges of an optimal 1-planar graph $G$, then the resulting graph is  denoted by $Q(G)$.  To prove Theorem \ref{thm1.15}, we first give a clear relationship between optimal 1-planar graphs and quadrangulations on the sphere as follows (see Theorem 11 in \cite{15}).

\begin{thm}[\cite{15}]\label{3.1} Let $H$ be a simple quadrangulation on the sphere. Then there
exists a simple optimal 1-planar graph $G$ such that $H = Q(G)$ if and only if $H$ is
3-connected.
\end{thm}

From Theorem \ref{3.1}, adding two diagonal edges within every face of a 3-connected quadrangulation on the sphere results in an optimal 1-planar graph. As an immediate consequence  we have the following result.

\begin{lem}\label{3.2} Let $G$ be an optimal 1-planar graph. Then for any $v\in V(G)$, all edges incident with $v$ in $G$ are alternately crossing and non-crossing edges in clockwise way. So the degree of each vertex in $G$ is even.
\end{lem}




\noindent\textbf{Proof of Theorem \ref{thm1.15}}~~Suppose to the contrary that there exists a 3-extendable optimal 1-planar graph $G$. By Lemma \ref{1.1}  $G$ is 4-connected. So  $G$ has  a vertex $v$ such that $4\leq {d}_{G}(v)\leq 7 $ by Lemma \ref{1.7}.

By Lemma \ref{3.2}  ${d}_{G}(v)$ is even, so ${d}_{G}(v)=4$ or 6. If ${d}_{G}(v)=4$, then by Lemma \ref{3.2}, $v$ has exactly two neighbors $v_{1}$ and $v_{3}$ in $Q(G)$ so that $G$ has one diagonal edge between $ v_{1}$ and $v_{3}$ in each 4-face on two sides of 3-path $v_1vv_{3}$.  Two edges joining $v_1$ and  $v_3$ clearly would become multiple edges,  contradicting that $G$ is simple.

If ${d}_{G}(v)=6$, let $vv_{i}$,  $1\leq i\leq 6$, be the consecutive edges incident with $v$ in counterclockwise order. Then by Lemma \ref{3.2}, without loss of generality suppose that $vv_{1}, vv_{3}$ and $vv_{5}$ are non-crossing edges, and $vv_{2}, vv_{4}$ and $vv_{6}$ are crossing edges (see Fig \ref{p29}). Further, $v$ is incident with exactly three 4-faces of $Q(G)$ so that their face cycles are $vv_1v_2v_3v$, $vv_3v_4v_5v$ and $vv_5v_6v_1v$, which implies that  $Q(G)$ has a 6-cycle $v_1v_2v_3v_4v_5v_6v_1$. Hence $G$ has  a 3-matching  $\{v_{1}v_{2},~v_{3}v_{4},~v_{5}v_{6}\}$ covering all neighbors of $v$,  which is    not extendable to a perfect matching of $G$, contradicting Lemma \ref{2.1} or the definition of 3-extendable graphs. Therefore, no optimal 1-planar graph is 3-extendable.

\qed
\begin{figure}[htb]
\centering
\includegraphics[height=4.5cm,width=4cm]{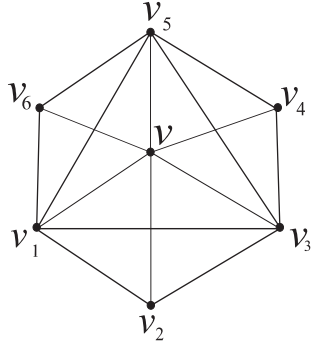}
\caption{\label{p29}A 6-vertex in an optimal 1-planar graph.}
\end{figure}

\section{Proof of Theorem \ref{thm1.14}}

To obtain our main result in this section  we first give some preliminaries. The following lemma  due to N. Dean \cite{9} will play  a crucial role in the proof of Theorem \ref{thm1.14}. For a graph $G$, we use $N(v)$ to denote the neighborhood of a vertex $v$ and $G[N(v)]$ for the induced subgraph of $G$ by $N(v)$.

\begin{lem}[\cite{9}]\label{2.1} Let $v$ be a vertex of degree $n+t$ in an $n$-extendable graph $G$. Then $G[N(v)]$ does not contain a matching of size $t$.
\end{lem}

We now describe some terminologies  and notations of a graph $G$.   A vertex is a $t$-vertex (resp., $t^{+}$-vertex, $t^{-}$-vertex) if ${d}_{G}(v)=t$ (resp., ${d}_{G}(v)\geq t$, ${d}_{G}(v)\leq t$). For a face $f$ of a connected plane graph $G$,
 the face degree of $f$, denoted by $d_G(f)$, is the length of the closed walk along the boundary of $f$. Such closed walk is called a {\em face walk}, and a {\em face cycle} whenever it is  cycle. Similarly  a $t$-face  (resp., $t^{+}$-face, and $t^{-}$-face) refers to a face with degree $t$ (resp., at least $t$, and at most $t$). 


The {\sl associated plane graph} $G^\times$ of a 1-planar graph $G$ is the plane graph that is obtained from $G$ by turning all crossings of $G$ into new vertices of degree four. These new vertices in $G^\times$ are called {\sl false vertices}, and the vertices of $G$ are called {\sl true vertices}.  A face in $G^\times$ is {\sl false} if it is incident with at least one false vertex; otherwise, it is {\sl true}.

Next we give a lemma which is used in proof of Theorem \ref{thm1.14}. This lemma shows that there exists a 1-matching in $G[N(v)]$ if a true vertex $v$ of the associated plane graph $G^\times$ of a 1-planar graph $G$ is incident with three consecutive false 3-faces.

\begin{lem}\label{2.2} Let $G^\times$ be the associated plane graph of a 1-planar graph $G$. If a true vertex $v$ is incident with three consecutive false 3-faces in $G^\times$, then $G[N(v)]$ contains an edge of $G$.
\end{lem}

\begin{proof}
 If a true vertex $v$ is incident with a false 3-face, then $v$ is adjacent to one true vertex and one false vertex on the false 3-face by 1-planarity of $G$. If $v$ is incident with three consecutive false 3-faces in $G^\times$, let $vv_{i}$, $1\le i\le 4$, denote four consecutive edges in $G^\times$ so that the $vv_iv_{i+1}v$ are false 3-faces for each $i=1,2,3$. Then $v_1,v_2, v_3$ and $v_4$ are false and true vertices in an alternative way.
If $v_{1}$ is a false vertex, then $v_{3}$ is a false vertex, $v_{2}$ and $v_{4}$ are true vertices and $v_{2}v_{4}$ is an edge of $G$ passing through $v_{3}$. Similarly, If $v_{1}$ is a true vertex, then $v_{3}$ is a true vertex and $v_{1}v_{3}$ is an edge of $G$ passing through $v_{2}$. Hence either $v_1v_3$ or $v_2v_4$ is an edge of $G$ (see Fig. \ref{p28}).
 That is, $G[N(v)]$ contains an edge.
 \end{proof}

\begin{figure}[htb]
\centering
\includegraphics[height=5cm,width=9cm]{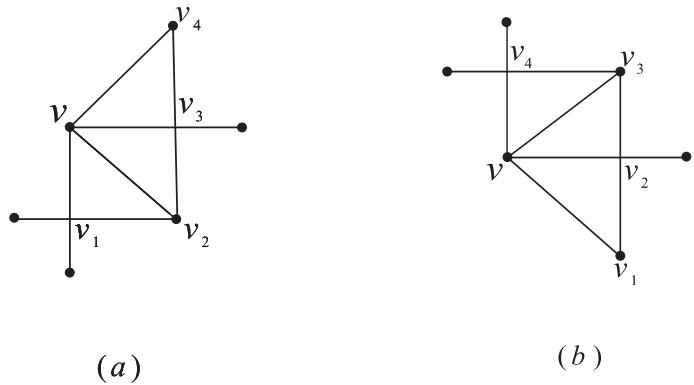}
\caption{\label{p28}A true vertex $v$ in $G^\times$ which is incident with three consecutive false 3-faces.}
\end{figure}


 \noindent\textbf{Proof of Theorem 1.11}~~Suppose to the contrary that there is a 5-extendable 1-planar graph $G$. By Lemma \ref{1.1}, $G$ is 6-connected and $\delta (G)\geq 6$. In the following, we will apply the discharging method on the associated plane graph $G^\times$.

By Euler's formula and degree-sum formulas:
\begin{equation}
|V(G^{\times})|-|E(G^{\times})|+|F(G^{\times})|=2
\end{equation}
\begin{equation}
\sum_{v\in V(G^{\times})}d_{G^{\times}}(v)=2|E(G^{\times})|,~\rm{and}
\end{equation}
\begin{equation}
\sum_{f\in F(G^{\times})}d_{G^{\times}}(f)=2|E(G^{\times})|,
\end{equation}

\noindent we have
\begin{eqnarray}
\sum_{v\in V(G^{\times})}(3d_{G^{\times}}(v)-10)+\sum_{f\in F(G^{\times})}(2d_{G^{\times}}(f)-10)=-20
\end{eqnarray}

First, we give an initial charge function:

\begin{itemize}
  \item $w(v) = 3{d}_{G^\times}(v)-10$, for each $v\in V(G^\times)$, and~
  \item $w(f) = 2{d}_{G^\times}(f)-10$, for each $f \in F(G^\times)$.
\end{itemize}

Next, we will design some discharging rules. Let $w'$ be the new charge after the discharging process. It suffices to show that $w'(x)\geq0$ for each $x\in {V(G^\times)\cup F(G^\times)}$, which leads to a contradiction with Eq. (3.4).

To get the target we now present four claims as follows.

\vspace{1ex}
\begin{raggedright}
 \textbf{Claim 1.}~~Every 6-vertex $v$ in $G^\times$ is not incident with a true 3-face, and is incident with at most four false 3-faces.
\end{raggedright}
\begin{proof}
 If $v$ is incident with a true 3-face $vv_{1}v_{2}$, then $v_{1}$ and $v_{2}$ are both true vertices and $v_{1}v_{2}$ is an edge of $G$. Here $v_{1}v_{2}$ is a 1-matching in $G[N(v)]$, contradicting Lemma \ref{2.1}. Thus every 6-vertex $v$ in $G^\times$ is not incident with a true 3-face.

 If $v$ is incident with at least five false 3-faces, then these five false 3-faces are consecutive. By Lemma \ref{2.2}, we can find a 1-matching in $G[N(v)]$, contradicting Lemma \ref{2.1}. Hence every 6-vertex $v$ in $G^\times$ is incident with at most four false 3-faces.
 \end{proof}

 \vspace{1ex}
\begin{raggedright}
 \textbf{Claim 2.}~~Every 7-vertex $v$ in $G^\times$ is incident with at most six false 3-faces. Moreover, if $v$ is incident with exactly six false 3-faces, then the other  face incident with $v$ is a $4^{+}$-face.
\end{raggedright}
\begin{proof}
 If $v$ is incident with seven false 3-faces, let  $vv_{i}$, $1\le i\le 7$, denote the seven consecutive edges in $G^\times$. By the proof of Lemma \ref{2.2} we have that  false vertices and true vertices alternate in the cyclic ordering $v_{1}, v_{2}, v_{3}, v_{4}, v_{5}, v_{6}, v_{7}, v_{1}$, which is obviously impossible. Hence every 7-vertex $v$ in $G^\times$ is incident with at most six false 3-faces.

 If $v$ is incident with exactly six false 3-faces, let $vv_{1}v_{2}, vv_{2}v_{3}, vv_{3}v_{4}, vv_{4}v_{5}, vv_{5}v_{6}, vv_{6}v_{7}$  denote these six consecutive false 3-faces. Then $v_1,v_2, v_3, v_4, v_5, v_6,v_7$ are false and true vertices in an alternative way. If $v_{1}$ is a true vertex (see Fig. \ref{p26}(a)), then $v_{3}, v_{5}$ and $v_{7}$ are all true vertices and $v_{2}, v_{4}$ and $v_{6}$ are all false vertices. So $v_{1}v_{3}$ is an edge of $G$ passing through $v_{2}$, and  $v_{5}v_{7}$ is an edge of $G$ passing through $v_{6}$. Then we find a 2-matching $\{v_{1}v_{3}, v_{5}v_{7}$\} in $G[N(v)]$, contradicting Lemma \ref{2.1}. Thus $v_{1}$ and $v_7$ are both  false vertices, and $v_2$ and $v_6$ are true vertices (see Fig. \ref{p26}(b)). By the 1-planarity the edges with end-vertices $v_2$ and $v_6$ in $G$ passing through $v_1$ and $v_7$ respectively are different. The other endvertices (may be the same) of them and vertices $v, v_{1}$ and $v_{7}$ all lie on the last one face incident with $v$, which must be a false $4^{+}$-face.
 \end{proof}

\begin{figure}[htp]
\centering
\includegraphics[height=5.5cm,width=11cm]{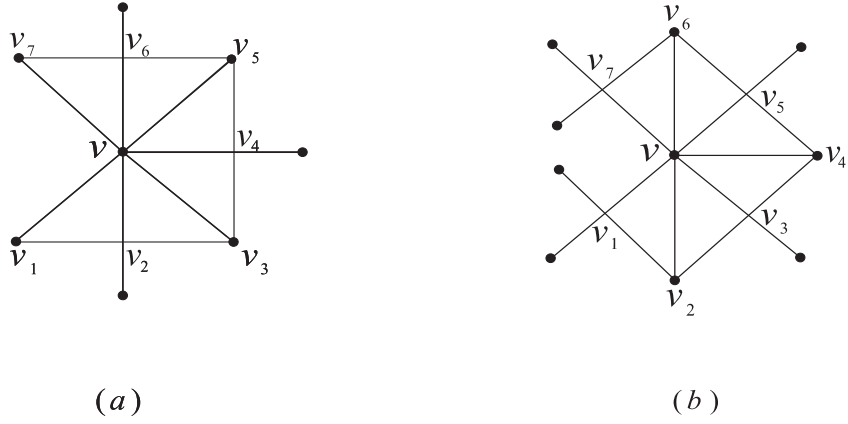}
\caption{\label{p26}7-vertex $v$ in $G^\times$ which is incident with exactly six false 3-faces.}
\end{figure}

 \vspace{1ex}
\begin{raggedright}
 \textbf{Claim 3.}~~Every 7-vertex $v$ in $G^\times$ is incident with at most two true 3-faces. Moreover, if $v$ is incident with  two true 3-faces, then such two true 3-faces are adjacent.
\end{raggedright}
\begin{proof}
 Obviously each true 3-face incident with $v$ has an edge  $G[N(v)]$. If $v$ is incident with two nonconsecutive true 3-faces, then  we can find a 2-matching in $G[N(v)]$ consisting of one edge in each such true 3-face, contradicting Lemma \ref{2.1}. Hence the assertion holds.
 \end{proof}

\vspace{1ex}
\begin{raggedright}
 \textbf{Claim 4.}~~Let $v$ be a 7-vertex of $G^\times$.  Then the following two statements hold.

(i) if $v$ is incident with one true 3-face, then $v$ is incident with at most four false 3-faces;

(ii) if $v$ is incident with two true 3-faces, then $v$ is incident with at most three false 3-faces.
\end{raggedright}

\begin{proof}
%
(i) Suppose that $v$ is incident with one true 3-face, say $f=vv_{1}v_{2}v$. Suppose to the contrary that $v$ is incident with at least five false 3-faces.

 If $v$ is incident with three consecutive false 3-faces such that they have no common edges with the true 3-face $vv_{1}v_{2}$. By Lemma \ref{2.2}, we can find an edge of $G[N(v)]$ among such three false 3-faces, which together with  edge $v_{1}v_{2}$ form a 2-matching in $G[N(v)]$,  contradicting Lemma \ref{2.1}.

 Otherwise, there are three consecutive false 3-faces, say $vv_2v_4, vv_4v_6, vv_6v_7$, incident with $v$ having a common edge $vv_{2}$ with $vv_{1}v_{2}$  and the other two consecutive false 3-faces, say $vv_1v_3$ and $vv_3v_5$, having a common edge $vv_{1}$ with $vv_{1}v_{2}$ (see Fig. \ref{p27}(a)). Since $v_1$ and $v_2$ are true vertices,  $v_{5}$ and $v_{6}$ are both true vertices, and $v_{1}v_{5}$ is an edge of $G$ passing through $v_{3}$ and $v_{2}v_{6}$ is an edge of $G$ passing through $v_{4}$. Now $\{v_{1}v_{5}, v_{2}v_{6}\}$ is a 2-matching in $G[N(v)]$, contradicting Lemma \ref{2.1}. This shows that Statement (i) holds.

 (ii) Suppose that  $v$ is incident with two true 3-faces.  Then such two true 3-faces are consecutive by Claim 3, which are denoted by $vv_{1}v_{2}$ and $vv_{2}v_{3}$. Suppose to the contrary that $v$ is incident with at least four false 3-faces.

 If there are  three consecutive false 3-faces incident with $v$, then they  have no common edges with at least one of the two true 3-face $vv_{1}v_{2}$ and $vv_{2}v_{3}$. Similar to (i)  we can find a 2-matching in $G[N(v)]$, contradicting Lemma \ref{2.1}.

 Otherwise, there are two consecutive false 3-faces, say $vv_1v_4$ and $vv_4v_6$, that has one common edge $vv_1$ with the true 3-face $vv_1v_2$, and the other two consecutive false 3-faces, say $vv_3v_5$ and $vv_5v_7$, that has one common edge $vv_3$ with the true 3-face $vv_2v_3$ (see Fig. \ref{p27}(b)).
Since $v_1$ and $v_3$ are true vertices,  $v_{6}$ and $v_{7}$ are both true vertices, and $v_{1}v_{6}$ is an edge of $G$ passing through $v_{4}$ and $v_{3}v_{7}$ is an edge of $G$ passing through $v_{5}$. Now $\{v_{1}v_{6}, v_{3}v_{7}\}$ is a 2-matching in $G[N(v)]$, contradicting Lemma \ref{2.1}. So Statement (ii) holds.
 \end{proof}

 \begin{figure}[htb]
\centering
\includegraphics[height=5cm,width=13cm]{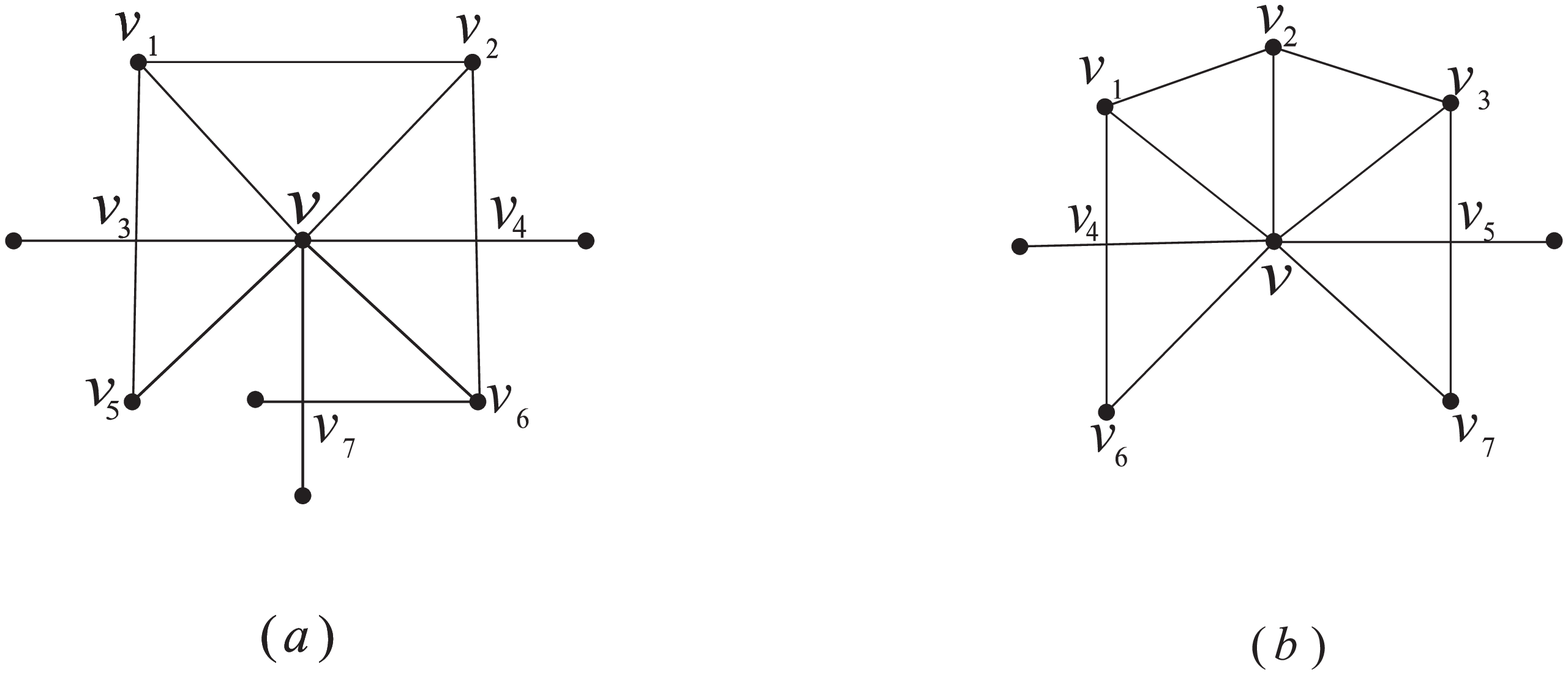}
\caption{\label{p27}7-vertex $v$ in $G^\times$ which is incident with at least one true 3-face.}
\end{figure}

 The following are the discharging rules.

 {\bf R1} Every false vertex in $G^\times$ gives $\frac{1}{2}$ to each incident $4^-$-face.

{\bf R2} Every $6^+$-vertex in $G^\times$ gives $\frac{4}{3}$ to each incident true 3-face, and gives $\frac{7}{4}$ to each incident false 3-face, and gives $\frac{1}{2}$ to each incident 4-face.

Next we verify that the new charge of each member in $ {V(G^\times)\cup F(G^\times)}$ is nonnegative. Firstly we consider any face $f\in F(G^\times)$. There are three cases according to the degree of a face.

{\bf Case 1.} ${d}_{G^\times}(f)=3$. If $f$ is a true 3-face,  then all vertices incident with $f$ are $6^+$-vertices since $\delta (G)\geq 6$. By R2, $w'(f)=2{d}_{G^\times}(f)-10+\frac{4}{3}\times 3=2\times3-10+4=0$. If $f$ is a false 3-face, then all vertices incident with $f$ are one false vertex and two $6^+$-vertices because $\delta (G)\geq 6$ and $G$ is a 1-planar graph. By R1 and R2, $w'(f)=2{d}_{G^\times}(f)-10+\frac{1}{2}+\frac{7}{4}\times2=2\times 3-10+\frac{1}{2}+\frac{7}{2}=0$.

{\bf Case 2.} ${d}_{G^\times}(f)=4$. Whether $f$ is a true 4-face or a false 4-face, by~R1~and~R2, $w'(f)=2{d}_{G^\times}(f)-10+\frac{1}{2}\times4=2\times4-10+2=0$.

{\bf Case 3.} ${d}_{G^\times}(f)\geq5$. By~R1~and~R2, $f$ has neither lost charge nor gained charge, so $w'(f)=2{d}_{G^\times}(f)-10\geq0$.

Now we  consider any vertex $v\in V(G^\times)$. There are four cases according to the degree of a vertex.

{\bf Case 4.} ${d}_{G^\times}(v)=4$. Because $v$ is incident with at most four $4^-$-faces, $w'(v)\geq3{d}_{G^\times}(v)-10-\frac{1}{2}\times4=3\times4-10-2=0$ by  R1.

{\bf Case 5.} ${d}_{G^\times}(v)=6$. By Claim 1 and R2, $w'(v)\geq3{d}_{G^\times}(v)-10-\frac{7}{4}\times4-\frac{1}{2}\times2=3\times6-10-7-1=0$.

{\bf Case 6.} ${d}_{G^\times}(v)=7$. If $v$ is not incident with a true 3-face, then by Claim 2 and R2, $w'(v)\geq3{d}_{G^\times}(v)-10-\frac{7}{4}\times6-\frac{1}{2}\times1=3\times7-10-\frac{21}{2}-\frac{1}{2}=0$; Otherwise, $v$ is incident with one or two true 3-faces by Claim 3. Further, by Claim 4 and R2, $w'(v)\geq3{d}_{G^\times}(v)-10-\frac{7}{4}\times4-\frac{4}{3}\times1-\frac{1}{2}\times2=3\times7-10-7-\frac{4}{3}-1=\frac{5}{3}>0$.

{\bf Case 7.} ${d}_{G^\times}(v)\geq8$. By R2, $w'(v)\geq3{d}_{G^\times}(v)-10-\frac{7}{4} {d}_{G^\times}(v)=\frac{5}{4} {d}_{G^\times}(v)-10\geq0$.

In summary we get that for each $x\in V(G^\times)\cup F(G^\times)$,
$w'(x)\geq0$, which is a contradiction. This completes the proof of Theorem \ref{thm1.14}.
\qed\\

We remark that  the non-5-extendability of 1-planar graphs in Theorem \ref{thm1.14} is  best possible by presenting a 4-extendable 1-planar graph drawn in Fig. \ref{p10}. We use a computer program to check the validation of the example: we find that the 1-planar graph has exactly 967469 4-matchings and each 4-matching can be contained  in a perfect matching. Further,  the graph has 1116948 perfect matchings and a 5-matching (see bold edges in Fig \ref{p10})  not extendable to a perfect matching.
\vspace{1cm}
\begin{figure}[htb]
\centering
\includegraphics[height=6cm,width=6.5cm]{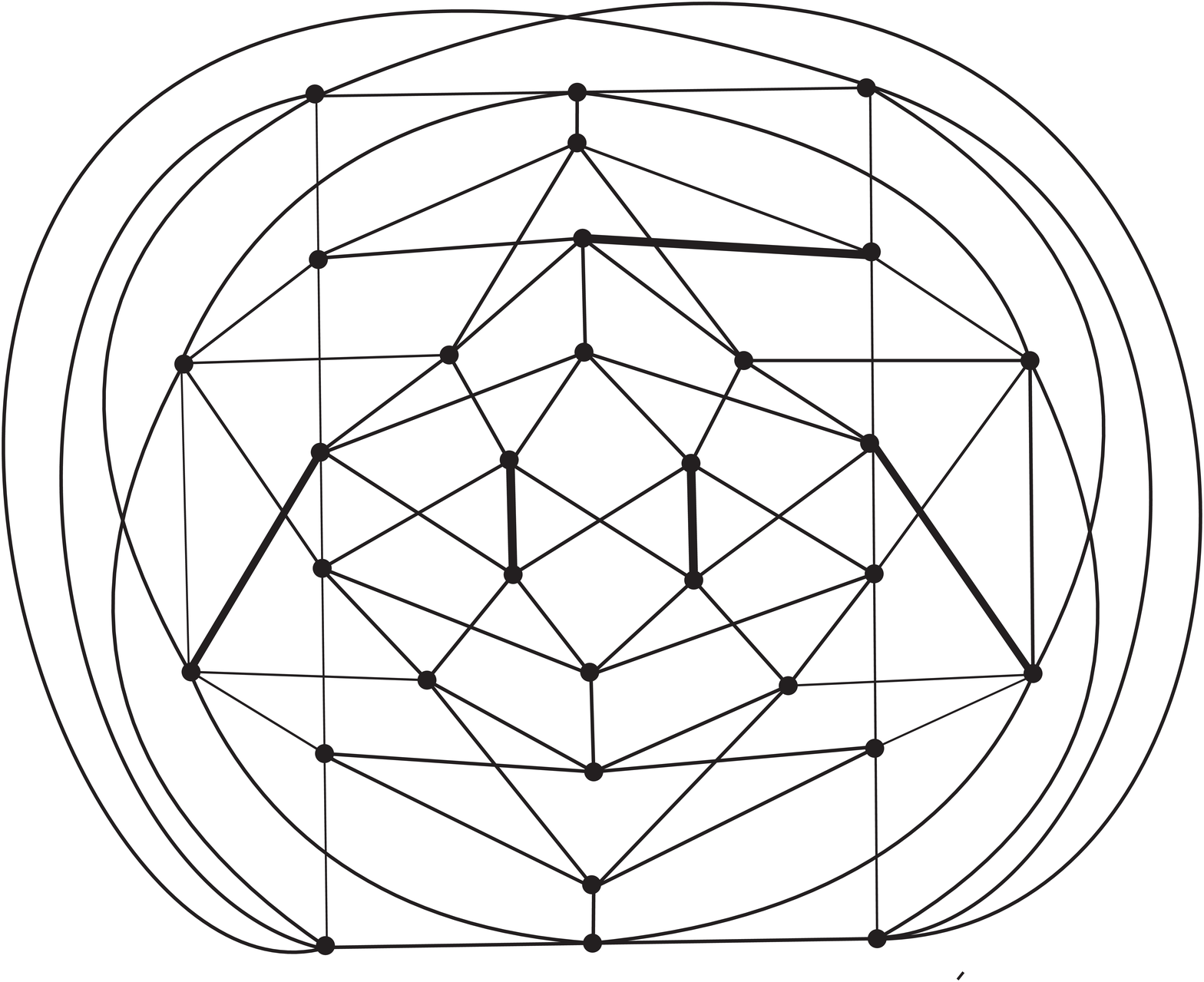}
\caption{\label{p10}A 4-extendable 1-planar graph.}
\end{figure}

\section{Remarks on  factor-criticality of 1-planar graphs}
We conclude with some remarks on  the maximum factor-criticality of (optimal) 1-planar graphs.  Yu \cite{19} and Favaron \cite{18} independently formulated the definition of a $k$-factor-critical graph. A graph of order $n$ is {\em $k$-factor-critical}, where $k$ is an integer with $0\leq k< n$ and $n+k$ is even, if $G-S$ admits a perfect matching for every set $S$ of $k$ vertices of $G$.  Favaron in \cite{18} obtained following basic properties of $k$-factor-critical graphs.

\begin{thm}[\cite{18}]\label{1.5}For $k\geq2$, any $k$-factor-critical graph of order $n> k$ is $(k-2)$-factor-critical.
\end{thm}

\begin{thm}[\cite{18}]\label{1.6}For $k\geq1$, any $k$-factor-critical graph of order $n> k$ is $k$-connected and $(k+1)$-edge-connected.
\end{thm}

\begin{thm}\label{thm1.17} No 1-planar graph is 7-factor-critical.
\end{thm}

\begin{proof}  Suppose to the contrary that there exists a 7-factor-critical 1-planar graph $G$. Then by Theorem \ref{1.6}, $G$ is 8-edge-connected. Then $\delta(G)\geq 8$, contradicting Lemma \ref{1.7}.
\end{proof}

\begin{figure}[h]
\centering
\includegraphics[height=5cm,width=5.5cm]{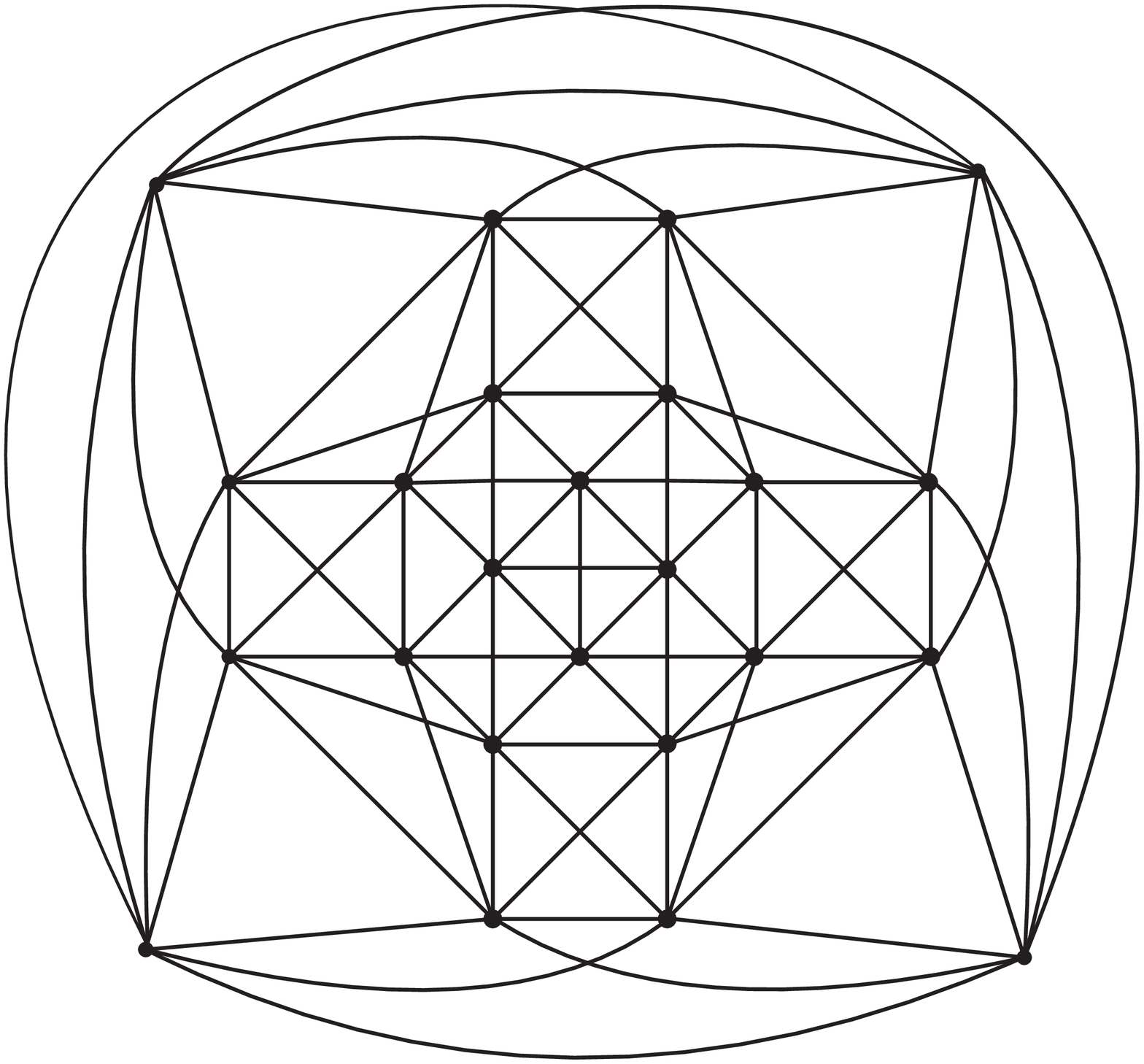}
\caption{\label{p44}A 6-factor-critical 1-planar graph.}
\end{figure}
\begin{figure}[h]
\centering
\includegraphics[height=5cm,width=5.5cm]{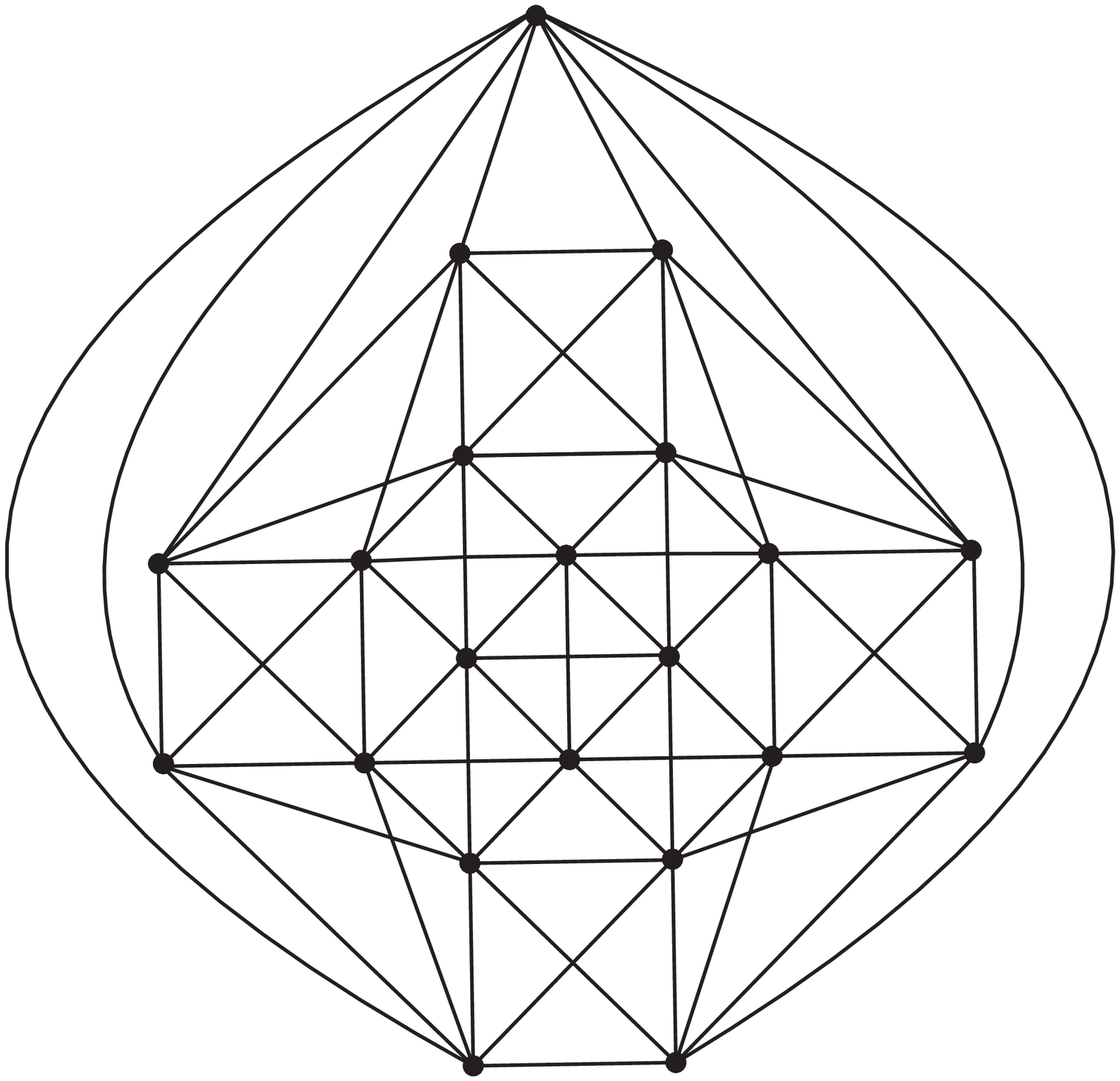}
\caption{\label{p42}A 5-factor-critical 1-planar graph.}
\end{figure}

In other word, each 1-planar graph is not 7-factor-critical.  The non-7-factor-criticallity of 1-planar graphs is  best possible by providing a 6-factor-critical 1-planar graph in Fig. \ref{p44} that is a 7-regular 1-planar graph taken from \cite{14} and a 5-factor-critical 1-planar graph of odd order in Fig. \ref{p42}.  We also present a computer  check to the validation of both examples: we find that the former has exactly 340361 perfect matchings and the removal of any 6 vertices results in a graph with a perfect matching, and  the removal of  any 5 vertices of the latter results in a graph with a perfect matching.

Next we turn to factor-criticality of optimal 1-planar graphs. The following theorem can be  obtained from Theorem \ref{thm1.15}. Here we give a direct proof.
\begin{thm}\label{thm1.17} No optimal 1-planar graph is 6-factor-critical.
\end{thm}

\begin{proof}  Suppose to the contrary that there exists a 6-factor-critical optimal 1-planar graph $G$. Then by Theorem \ref{1.6}, $G$ is 7-edge-connected and $\delta(G)\geq 7$. From Lemma \ref{3.2}, each vertex of $G$ has  even degree, so $\delta(G)\geq 8$, contradicting Lemma \ref{1.7}.
\end{proof}

Here the non-6-factor-criticallity of optimal 1-planar graphs is  best possible by presenting 4- and 5-factor-critical optimal 1-planar graphs shown in Fig. \ref{p13} (a) and (b) respectively.  Their validation has also been confirmed by a computer program.\\

\begin{figure}[h]
\centering
\includegraphics[height=7cm,width=12cm]{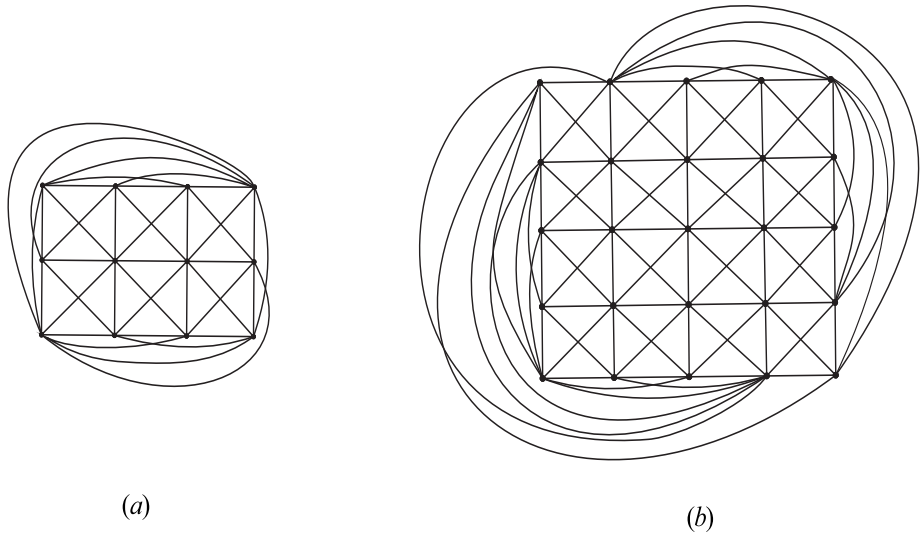}
\caption{\label{p13} 4- and 5-factor-critical optimal 1-planar graphs.}
\end{figure}

%
%
%

\end{document}